\documentclass{amsart}
\usepackage{amssymb, amsmath, mathrsfs, verbatim, url, bbm, bbold}

\makeatletter
\@namedef{subjclassname@2020}{%
  \textup{2020} Mathematics Subject Classification}
\makeatother

\theoremstyle{plain}
\newtheorem{theorem}{Theorem}[section]
\newtheorem{lemma}[theorem]{Lemma}
\newtheorem{proposition}[theorem]{Proposition}
\newtheorem{corollary}[theorem]{Corollary}

\theoremstyle{definition}

\newtheorem{question}[theorem]{Question}
\newtheorem{conjecture}[theorem]{Conjecture}

\newcommand{\re}{\upharpoonright}

\newcommand{\coord}{\mathsf{coord}}
\newcommand{\length}{\mathsf{length}}
\newcommand{\cl}{\mathsf{cl}}
\newcommand{\add}{\mathsf{add}}
\newcommand{\CDH}{\mathsf{CDH}}
\newcommand{\MP}{\mathsf{MP}}
\newcommand{\ZFC}{\mathsf{ZFC}}
\newcommand{\Cof}{\mathsf{Cof}}
\newcommand{\Fin}{\mathsf{Fin}}
\newcommand{\Ss}{\mathcal{S}}
\newcommand{\DD}{\mathcal{D}}
\newcommand{\FF}{\mathcal{F}}
\newcommand{\PP}{\mathcal{P}}
\newcommand{\XX}{\mathcal{X}}
\newcommand{\PPP}{\mathbb{P}}
\newcommand{\RRR}{\mathbb{R}}
\newcommand{\cccc}{\mathfrak{c}}
\newcommand{\pppp}{\mathfrak{p}}
\newcommand{\tttt}{\mathfrak{t}}

\begin{document}

\title[Non-meager $\mathsf{P}$-filters]{Non-meager $\mathsf{P}$-filters, Miller-measurability, and a question of Hru\v{s}\'{a}k}

\author{Andrea Medini}
\address{Institut f\"{u}r Diskrete Mathematik und Geometrie
\newline\indent Technische Universit\"{a}t Wien
\newline\indent Wiedner Hauptstra\ss e 8-10/104
\newline\indent 1040 Vienna, Austria}
\email{andrea.medini@tuwien.ac.at}
\urladdr{https://www.dmg.tuwien.ac.at/medini/}



\thanks{The author is grateful to Yurii Khomskii for directing him to the article \cite{goldstern_repicky_shelah_spinas}. This research was funded in whole by the Austrian Science Fund (FWF) DOI 10.55776/P35588. For open access purposes, the author has applied a CC BY public copyright license to any author-accepted manuscript version arising from this submission.}

\date{January 25, 2026}

\begin{abstract}
Given a cardinal $\kappa$ and filters $\FF_\alpha$ on $\omega$ for $\alpha\in\kappa$, we will show that if $\prod_{\alpha\in\kappa}\FF_\alpha$ is countable dense homogeneous then $\kappa<\pppp$ and each $\FF_\alpha$ is a non-meager $\mathsf{P}$-filter. This partially answers a question of Michael Hru\v{s}\'{a}k. Along the way, we will show that the product of fewer than $\pppp$ non-meager $\mathsf{P}$-filters has the Miller property. We will also describe explicitly the connection between Miller-measurability and the Miller property. As a corollary, we will see that the intersection of fewer than $\add(m^0)$ non-meager $\mathsf{P}$-filters is a non-meager $\mathsf{P}$-filter, where $m^0$ denotes the ideal of Miller-null sets. We will conclude by investigating the preservation of the Miller property under intersections and products.
\end{abstract}

\maketitle

\tableofcontents

\section{Introduction}\label{section_introduction}

By \emph{space} we will mean Tychonoff topological space. By \emph{countable} we will mean at most countable. Recall that a space $X$ is \emph{countable dense homogeneous} ($\CDH$) if $X$ is separable and for every pair $(D,E)$ of countable dense subsets of $X$ there exists a homemorphism $h:X\longrightarrow X$ such that $h[D]=E$. We refer to \cite[Sections 14-16]{arhangelskii_van_mill} for an introduction to countable dense homogeneity, and to \cite{hrusak_van_mill} for a collection of open problems on this topic. It is a fundamental result that every ``sufficiently homogeneous'' Polish space (such as $2^\omega$, $\omega^\omega$, or any separable metrizable manifold) is $\CDH$ (see \cite[Theorem 5.2]{anderson_curtis_van_mill} or \cite[Theorem 14.1]{arhangelskii_van_mill} for a precise statement).

Motivated by a question\footnote{\,We are referring to \cite[Question 3.2]{hrusak_zamora_aviles}, which asks for a non-Polish zero-dimensional separable metrizable space $X$ such that $X^\omega$ is $\CDH$; a $\ZFC$ example of such a space was eventually given in \cite[Theorem 8.4]{medini}.} from \cite{hrusak_zamora_aviles}, the article \cite{medini_milovich} began the study of the countable dense homogeneity of the (ultra)filters on $\omega$, viewed as subspaces of $2^\omega$. Further results on this topic were obtained in \cite{hernandez_gutierrez_hrusak} and \cite{repovs_zdomskyy_zhang}. The following result (which is a fragment of \cite[Theorem 10]{kunen_medini_zdomskyy}) can be viewed as the culmination of this line of research. According to \cite[Definition 5]{kunen_medini_zdomskyy}, a space $X$ has the \emph{Miller property} ($\MP$) if for every countable crowded $Q\subseteq X$ there exists a crowded $Q'\subseteq Q$ such that the closure of $Q'$ in $X$ is compact. This property was introduced because it is very convenient when countable dense homogeneity and filters are involved, and this fact will reassert itself in the present context.

\begin{theorem}[Kunen, Medini, Zdomskyy]\label{theorem_kunen_medini_zdomskyy}
Let $\FF$ be a filter. Then the following conditions are equivalent:
\begin{itemize}
\item $\FF$ is a non-meager $\mathsf{P}$-filter,
\item $\FF$ is $\CDH$,
\item $\FF$ has the $\MP$.
\end{itemize}
\end{theorem}

Given that the product of at most $\cccc$ separable spaces is separable by the classical Hewitt-Marczewski-Pondiczery theorem (see \cite[Corollary 2.3.16]{engelking} or \cite[Exercise III.2.13]{kunen}), one might wonder whether the $\kappa$-th power of a sufficiently nice separable space could be $\CDH$ for uncountable values of $\kappa$. This was consistently achieved by Stepr\={a}ns and Zhou \cite[Theorem 4]{steprans_zhou}, who showed that $2^\kappa$, $\RRR^\kappa$ and $[0,1]^\kappa$ are $\CDH$ for every infinite $\kappa<\pppp$, where $\pppp$ denotes the \emph{pseudointersection number} (see \cite[Definition III.1.21]{kunen}). In fact, in the recent article \cite{medini_steprans}, we showed that much more general results hold (not just for powers, but for suitable products as well). Furthermore, the cardinal $\pppp$ is sharp, because by \cite[Theorem 8.1]{medini_steprans} no product of at least $\pppp$ non-trivial separable metrizable spaces is $\CDH$.\footnote{\,This result generalizes \cite[Theorem 3.3]{hrusak_zamora_aviles}, and its proof uses essentially the same argument.}

However, the following question remains open (see \cite[Question 7.1]{hernandez_gutierrez}, where it is credited to Michael Hru\v{s}\'{a}k).

\begin{question}[Hru\v{s}\'{a}k]\label{question_hrusak}
For which cardinals $\kappa$ and filters $\FF$ is $\FF^\kappa$ countable dense homogeneous?
\end{question}

Given the history that we discussed above, the following seems natural.

\begin{conjecture}\label{conjecture_complete_answer}
Let $\kappa$ be a cardinal and let $\FF_\alpha$ be filters for $\alpha\in\kappa$. Then the following conditions are equivalent:
\begin{itemize}
\item $\prod_{\alpha\in\kappa}\FF_\alpha$ is $\CDH$,
\item $\kappa <\pppp$ and each $\FF_\alpha$ is a non-meager $\mathsf{P}$-filter.
\end{itemize}
\end{conjecture}

We will give a proof of one half of the above conjecture (see Theorem \ref{theorem_partial_answer}). Initially, we were rather confident that the other half could be proved by suitable adaptations of techniques from \cite{hernandez_gutierrez_hrusak} and \cite{medini_steprans}; unfortunately, unsurmountable technical difficulties ensued.

As a consequence of a key lemma, we will see that a product of fewer than $\pppp$ non-meager $\mathsf{P}$-filters has the $\MP$ (see Corollary \ref{corollary_generalized_miller}). This prompted us to investigate more generally the preservation of the $\MP$ under intersections and products, which will be accomplished in Sections \ref{section_preservation_intersections} and \ref{section_preservation_products}. As a preliminary to this investigation, we will describe explicitly the connection between the $\MP$ and the classical notion of Miller-measurability (see Proposition \ref{proposition_characterization_miller}). A particularly quotable result is that the intersection of fewer than $\add(m^0)$ non-meager $\mathsf{P}$-filters is a non-meager $\mathsf{P}$-filter, where $m^0$ denotes the ideal of Miller-null sets (see Corollary \ref{corollary_intersection_filters}).

\section{General preliminaries}\label{section_general_preliminaries}

Our reference for set theory is \cite{kunen}. In particular, we will assume familiarity with the cardinals $\pppp$ and $\tttt$, as well as Martin's Axiom. Given $x,y\subseteq\omega$, we will write $x\subseteq^\ast y$ to mean that $x\setminus y$ is finite; we will also write $x=^\ast y$ to mean that $x\subseteq^\ast y$ and $y\subseteq^\ast x$. Given $\XX\subseteq\PP(\omega)$, in accordance with \cite[Definition II.1.21]{kunen}, we will say that $z\subseteq\omega$ is a \emph{pseudointersection} of $\XX$ if $z$ is infinite and $z\subseteq^\ast x$ for every $x\in\XX$. We will use the notation
$$
\Fin=\{x\subseteq\omega:|x|<\omega\}\text{ and }\Cof=\{x\subseteq\omega:|\omega\setminus x|<\omega\}.
$$

Our reference for general topology is \cite{engelking}. We will also assume some familiarity with the basic theory of Polish spaces and their analytic subsets, for which we refer to \cite{kechris}. A space is \emph{crowded} if it is non-empty and it has no isolated points. A space is \emph{scattered} if it has no crowded subspaces. A space is \emph{zero-dimensional} if it is non-empty, and it has a base consisting of clopen sets.\footnote{\,The empty space has dimension $-1$ (see \cite[Section 7.1]{engelking}).} A \emph{compactification} of a space $X$ is a compact space $\gamma X$ such that $X$ is a dense subspace of $\gamma X$. A subset $X$ of a space $Z$ is \emph{meager} if there exist closed subsets $C_n$ of $Z$ with empty interior for $n\in\omega$ such that $X\subseteq\bigcup_{n\in\omega}C_n$. A space $X$ is \emph{meager} if $X$ is a meager subset of $X$. A space $X$ is \emph{Baire} if every non-empty open subspace of $X$ is non-meager. Given spaces $Z$ and $X$, we will say that $X'$ is a \emph{copy of $X$ in $Z$} if $X'$ is a subspace of $Z$ that is homeomorphic to $X$.

The following result is essentially due to Dobrowolski, Krupski and Marciszewski; in fact, it can be obtained by making obvious modifications to the proof of \cite[Theorem 4.1]{dobrowolski_krupski_marciszewski}. Theorem \ref{theorem_dobrowolski_krupski_marciszewski} will be useful in the proof of Theorem \ref{theorem_partial_answer}.

\begin{theorem}\label{theorem_dobrowolski_krupski_marciszewski}
Let $\kappa$ be an infinite cardinal, and let $X_\alpha$ for $\alpha\in\kappa$ be separable spaces such that each $|X_\alpha |\geq 2$. Set $X=\prod_{\alpha\in\kappa}X_\alpha$. If $X$ is $\CDH$ then $X$ is a Baire space.
\end{theorem}

The following two well-known results will be very useful in the final portion of this article (see \cite[Theorem 7.7 and Corollary 21.23]{kechris} respectively).

\begin{theorem}[Alexandrov, Urysohn]\label{theorem_characterization_baire_space}
Let $X$ be a space. Then the following conditions are equivalent:
\begin{itemize}
\item $X$ is homeomorphic to $\omega^\omega$,
\item $X$ is zero-dimensional, Polish, and no non-empty open subspace of $X$ is compact.
\end{itemize}	
\end{theorem}

\begin{theorem}[Kechris, Saint Raymond]\label{theorem_kechris_saint_raymond}
Let $Z$ be a Polish space, and let $X\subseteq Z$ be analytic. Then exactly one of the following conditions holds:
\begin{itemize}
\item $X$ is contained in a $\sigma$-compact subspace of $Z$,
\item $X$ contains a closed copy of $\omega^\omega$ in $Z$.
\end{itemize}
\end{theorem}

Finally, we state a simple lemma which will be convenient in certain proofs that involve the $\MP$ (it can be proved exactly like \cite[Lemma 3.2]{medini_zdomskyy}).

\begin{lemma}\label{lemma_crowded}
Let $K$ be a compact metrizable space, and let $D$ be a dense subspace of $K$. If $N$ is a closed copy of $\omega^\omega$ in $K\setminus D$ then $\cl(N)\cap D$ is crowded, where $\cl$ denotes closure in $K$.
\end{lemma}

\section{Preliminaries concerning filters}\label{section_preliminaries_filters}

Recall that a \emph{filter} is a collection $\FF\subseteq\PP(\omega)$ that satisfies the following conditions:
\begin{itemize}
\item $\varnothing\notin\FF$ and $\omega\in\FF$,
\item If $x\in\FF$ and $x=^\ast y\subseteq\omega$ then $y\in\FF$,
\item If $x\in\FF$ and $x\subseteq y\subseteq\omega$ then $y\in\FF$,
\item If $x,y\in\FF$ then $x\cap y\in\FF$.
\end{itemize}
We will use the same terminology when we consider a \emph{filter} on a forcing poset $\PPP$, in which case however we refer to \cite[Definition III.3.10]{kunen}.

As usual, we will identify every subset of $\omega$ with its characteristic function, so that every $\XX\subseteq\PP(\omega)$ will be endowed with the subspace topology that it inherits from $2^\omega$. Naturally, the relevant case here is when $\XX=\FF$ is a filter. In particular, by \emph{non-meager filter} we mean a filter $\FF$ that is a non-meager space with this topology. Since every filter is dense in $2^\omega$, this is equivalent to $\FF$ being a non-meager subset of $2^\omega$. The following well-known characterization of non-meager filters originally appeared as part of \cite[Th\'{e}or\`{e}me 21]{talagrand}.

\begin{lemma}[Talagrand]\label{lemma_talagrand}
Let $\FF$ be a filter. Then the following conditions are equivalent:
\begin{itemize}
\item $\FF$ is non-meager,
\item Given pairwise disjoint $a_n\in\Fin$ for $n\in\omega$, there exist $z\in\FF$ and $I\in [\omega]^\omega$ such that $z\cap a_n=\varnothing$ for every $n\in I$.
\end{itemize}
\end{lemma}

In fact, given a sufficiently small family of non-meager filters, a single $I$ can be found that works for every filter. This is the content of the following result, which will be useful in the proof of Lemma \ref{lemma_generalized_miller}.

\begin{lemma}\label{lemma_generalized_talagrand}
Let $\kappa<\tttt$ be a cardinal, and let $\FF_\alpha$ be non-meager filters for $\alpha\in\kappa$. Then, given pairwise disjoint $a_n\in\Fin$ for $n\in\omega$, there exist $z_\alpha\in\FF_\alpha$ for $\alpha\in\kappa$ and $I\in [\omega]^\omega$ such that $z_\alpha\cap a_n=\varnothing$ for every $\alpha\in\kappa$ and $n\in I$.
\end{lemma}
\begin{proof}
Let $a_n\in\Fin$ for $n\in\omega$ be pairwise disjoint. Using Lemma \ref{lemma_talagrand} together with the assumption $\kappa<\tttt$, one can easily construct $w_\alpha\in\FF_\alpha$ and $I_\alpha\in [\omega]^\omega$ for $\alpha\in\kappa$ such that the following conditions are satisfied:
\begin{itemize}
\item $w_\alpha\cap a_n=\varnothing$ for every $\alpha\in\kappa$ and $n\in I_\alpha$,
\item $I_\beta\subseteq^\ast I_\alpha$ whenever $\alpha\leq\beta$.
\end{itemize}

Once the construction is completed, let $I$ be a pseudointersection of $\{I_\alpha:\alpha\in\kappa\}$. This means that we can fix $n_\alpha\in\omega$ for $\alpha\in\kappa$ such that $w_\alpha\cap a_n =\varnothing$ for every $n\in I\setminus n_\alpha$. Finally, set
$$
z_\alpha=w_\alpha\setminus\bigcup_{n\in n_\alpha}a_n.
$$
for $\alpha\in\kappa$. It is easy to realize that the $z_\alpha$ and $I$ are as desired.
\end{proof}

A filter $\FF$ is a \emph{$\mathsf{P}$-filter} if every countable $\XX\subseteq\FF$ has a pseudointersection in $\FF$. Given $z\in [\omega]^\omega$, the filter $\{x\subseteq\omega:z\subseteq^\ast x\}$ is a trivial example of (meager) $\mathsf{P}$-filter. While it is not hard to construct non-meager $\mathsf{P}$-filters under additional set-theoretic assumptions, their existence in $\ZFC$ is a long-standing open problem (see \cite{just_mathias_prikry_simon}).

The following result is essentially due to Shelah; in fact, it is a slight generalization of \cite[Fact 4.3 on page 327]{shelah}, and it can be proved in the same way (see also \cite[Lemma 2.5]{hernandez_gutierrez_hrusak}). The first part of Proposition \ref{proposition_generalized_shelah} (which is a trivial exercise) will be useful in the proof of Theorem \ref{theorem_partial_answer}, while the full result will be relevant in Section \ref{section_preservation_products} (see the remark that follows Theorem \ref{theorem_product_miller}).

\begin{proposition}\label{proposition_generalized_shelah}
Let $\kappa\leq\omega$ be a non-zero cardinal, and let $\FF_\alpha$ be filters for $\alpha\in\kappa$. Set $X=\prod_{\alpha\in\kappa}\FF_\alpha$. Then $X$ is homeomorphic to a filter. Furthermore, if each $\FF_\alpha$ is a non-meager $\mathsf{P}$-filter then $X$ is homeomorphic to a non-meager $\mathsf{P}$-filter.
\end{proposition}

\section{Generalizing an old lemma}\label{section_old_lemma}

The following result is a ``multivariable'' generalization of \cite[Lemma 7]{kunen_medini_zdomskyy}; it is the combinatorial core of the proof of Theorem \ref{theorem_partial_answer}. Given $z\subseteq\omega$, we will use the notation $z\!\uparrow\,\,=\{x\subseteq\omega:z\subseteq x\}$. Observe that each $z\!\uparrow$ is a compact subspace of $2^\omega$.

\begin{lemma}\label{lemma_generalized_miller}
Let $\kappa<\pppp$ be a cardinal, and let $\FF_\alpha$ be non-meager filters for $\alpha\in\kappa$. Set $X=\prod_{\alpha\in\kappa}\FF_\alpha$, and denote the projections by $\pi_\alpha:X\longrightarrow\FF_\alpha$ for $\alpha\in\kappa$. Assume that $Q\subseteq X$ is countable and crowded, and that each $\pi_\alpha[Q]$ has a pseudointersection in $\FF_\alpha$. Then there exists a crowded $Q'\subseteq Q$ and $z_\alpha\in\FF_\alpha$ for $\alpha\in\kappa$ such that each $\pi_\alpha[Q']\subseteq z_\alpha\!\uparrow$.
\end{lemma}
\begin{proof}
Start by choosing an arbitrary $r_\varnothing\in Q$. Fix $v_\alpha\in\FF_\alpha$ for $\alpha\in\kappa$ such that $v_\alpha\subseteq^\ast\pi_\alpha(r)$ for all $r\in Q$ and $\alpha\in\kappa$. Also assume that each $v_\alpha\subseteq\pi_\alpha(r_\varnothing)$. Define
$$
\Ss=\{\varnothing\}\cup\{s\in 2^{<\omega}:|s|\geq 1\text{ and
}s(|s|-1)=0\}.
$$
Also set $\Ss_{\leq n}=\{s\in\Ss :|s|\leq n\}$ and $\Ss_n=\{s\in\Ss :|s|=n\}$ for $n\in\omega$. We will need a bookkeeping function $\psi:\omega\longrightarrow\omega$ that satisfies the following requirements:
\begin{itemize}
\item[(1)] $\psi(i)<i$ for every $i\geq 1$, 
\item[(2)] $\psi^{-1}(j)$ is infinite for every $j\in\omega$.
\end{itemize}
Constructing such a function is an easy exercise, left to the reader.

Denote by $\PPP$ the set of all quadruples of the form $p=(n,F,\zeta,\vec{k})=(n^p,F^p,\zeta^p,\vec{k}^p)$ that satisfy the following requirements:
\begin{itemize}
\item[(3)] $n\in\omega$,
\item[(4)] $F\in [\kappa]^{<\omega}$,
\item[(5)] $\zeta:\Ss_{\leq n}\longrightarrow Q$ and $\zeta(\varnothing)=r_\varnothing$,
\item[(6)] $s\subsetneq s'$ implies $\zeta(s)\neq\zeta(s')$ for all $s,s'\in\Ss_{\leq n}$,
\item[(7)] $\vec{k}=(k_0,\ldots,k_n)\in\omega^{n+1}$, where $0=k_0<\cdots<k_n$.
\end{itemize}
We will also use the obvious notation $\vec{k}^p=(k^p_0,\ldots,k^p_n)$ for $p\in\PPP$, where $n=n^p$. Define $\ell_s\in\omega$ for every $s\in 2^{<\omega}$ so that $\{t\in\Ss:t\subsetneq s\}=\{t^s_i:i<\ell_s\}$, where $\varnothing=t^s_0\subsetneq\cdots\subsetneq t^s_{\ell_s-1}$ whenever $\ell_s\geq 1$. Notice that if $s,s'\in 2^{<\omega}$ and $s\subseteq s'$ then $t^s_i=t^{s'}_i$ for every $i<\ell_s$.

Order $\PPP$ by declaring $q\leq p$ if the following conditions are satisfied:
\begin{itemize}
\item[(8)] $n^q\geq n^p$,
\item[(9)] $F^q\supseteq F^p$,
\item[(10)] $\zeta^q\supseteq\zeta^p$,
\item[(11)] $\vec{k}^q$ is an end-extension of $\vec{k}^p$ (that is, $k^q_i=k^p_i$ for every $i\leq n^p$),
\item[(12)] $\zeta^q(s)\re\big(F^p\times k^q_{|s|-1}\big)=\zeta^q\big(t^s_{\psi(\ell_s)}\big)\re\big(F^p\times k^q_{|s|-1}\big)$ for every $s\in\Ss$ such that $n^p<|s|\leq n^q$,
\item[(13)] $v_\alpha\setminus k^q_n\subseteq\bigcap\big\{\pi_\alpha\big(\zeta^q(s)\big):s\in\Ss_{\leq n}\big\}$ whenever $\alpha\in F^p$ and $n^p<n\leq n^q$.
\end{itemize}

Given $\alpha\in\kappa$ and $n\in\omega$, make the following definitions:
\begin{itemize}
\item $D^\coord_\alpha=\{p\in\PPP:\alpha\in F^p\}$,
\item $D^\length_n=\{p\in\PPP:n^p\geq n\}$.
\end{itemize}

\noindent\textbf{Claim 1.} Each $D^\coord_\alpha$ is a dense subset of $\PPP$.

\noindent\textit{Proof.} This is trivial. $\blacksquare$

\noindent\textbf{Claim 2.} Each $D^\length_n$ is a dense subset of $\PPP$.

\noindent\textit{Proof.} Pick $n\in\omega$. Clearly, it will be enough to show that for every $p\in\PPP$ there exists $q\in\PPP$ such that $q\leq p$ and $n^q=n^p+1$. So pick $p=(n,F,\zeta,\vec{k})\in\PPP$. Begin to construct $q$ by setting $n^q=n+1$ and $F^q=F$. In order to satisfy condition $(10)$, we must set $\zeta^q(s)=\zeta(s)$ for $s\in\Ss_{\leq n}$. Given $s\in\Ss_{n+1}$, using condition $(1)$ and the assumption that $Q$ is crowded, we can define $\zeta^q(s)$ to be any element of
$$
\big\{r\in Q:r\re (F\times k_n)=\zeta\big(t^s_{\psi(\ell_s)}\big)\re (F\times k_n)\big\}\setminus\{\zeta(t^s_i):i<\ell_s\}.
$$
This way, conditions $(6)$ and $(12)$ will be satisfied. In order to satisfy $(11)$, we must set $k^q_i=k^p_i$ for $i\leq n$. To conclude the construction of $q$, simply choose $k^q_{n+1}>k^p_n$ big enough so that condition $(13)$ will be satisfied. $\blacksquare$

\noindent\textbf{Claim 3.} $\PPP$ is $\sigma\text{-}\mathrm{centered}$.

\noindent\textit{Proof.} Notice that there are only countably many triples $(n,\zeta,\vec{k})$ such that the requirements $(3)$, $(5)$, $(6)$ and $(7)$ are satisfied. Given such a triple, set
$$
\PPP(n,\zeta,\vec{k})=\{p\in\PPP:n^p=n\text{, }\zeta^p=\zeta\text{ and }\vec{k}^p=\vec{k}\}.
$$
Since each $\PPP(n,\zeta,\vec{k})$ is clearly centered, the desired conclusion follows. $\blacksquare$

Observe that the collection
$$
\DD=\{D^\coord_\alpha:\alpha\in\kappa\}\cup\{D^\length_n:n\in\omega\}
$$
consists of dense subsets of $\PPP$ by Claims 1 and 2, and that $|\DD|<\pppp$ because $\kappa<\pppp$. Therefore, since $\PPP$ is $\sigma$-centered by Claim 3, it is possible to apply Bell's Theorem\footnote{\,This result was essentially obtained in \cite{bell}; however, the modern statement given in \cite{kunen} is the one that is needed here.} (see \cite[Theorem III.3.61]{kunen}), which guarantees the existence of a filter $G$ on $\PPP$ that intersects every element of $\DD$.

Define
$$
\zeta=\bigcup\{\zeta^p:p\in G\},
$$
and observe that $\zeta:\Ss\longrightarrow Q$. Then set
$$
Q'=\{\zeta(\phi\re n):n\in\omega\text{ such that }\phi\re n\in\Ss\}\subseteq Q.
$$

Given $n\in\omega$, pick any $p\in G\cap D^\length_n$, then set $k_n=k^p_n$. Since each $\FF_\alpha$ is a non-meager filter, an application of Lemma \ref{lemma_generalized_talagrand} yields $w_\alpha\in\FF_\alpha$ for $\alpha\in\kappa$ and a function $\phi:\omega\longrightarrow 2$ such that $\phi^{-1}(0)$ is infinite and
$$
w_\alpha\cap\bigcup\{[k_n,k_{n+1}):n\in\phi^{-1}(0)\}=\varnothing.
$$
Fix $p_\alpha\in G\cap D^\coord_\alpha$ for $\alpha\in\kappa$. Then choose $n_\alpha\in\omega$ large enough for $\alpha\in\kappa$ so that each
$$
(v_\alpha\cap w_\alpha)\setminus n_\alpha\subseteq\bigcap\big\{\pi_\alpha\big(\zeta(s)\big):s\in\Ss_{\leq n^{p_\alpha}}\big\}.
$$
Finally, set $z_\alpha=(v_\alpha\cap w_\alpha)\setminus n_\alpha\in\FF_\alpha$ for $\alpha\in\kappa$. It is clear that the following two claims will conclude the proof.

\noindent\textbf{Claim 4.} $Q'$ is crowded.

\noindent\textit{Proof.} Pick $r\in Q'$. Let $F\in [\kappa]^{<\omega}$ and $\ell\in\omega$. We will find $r'\in Q'$ such that $r'\neq r$ and $r'\re (F\times\ell)=r\re (F\times\ell)$. Fix $n\geq 1$ such that $k_{n-1}\geq\ell$. Using Claims 1 and 2, it is possible to find $p\in G$ such that $n^p\geq n$ and $F^p\supseteq F$. Let $s\in\Ss$ be such that $s\subseteq\phi$ and $\zeta(s)=r$. Let $j\in\omega$ be the unique index such that $s=t^{s'}_j$ for every $s'\in 2^{<\omega}$ such that $s\subsetneq s'$. By condition $(2)$, it is possible to find $s'\in\Ss$ such that $s\subsetneq s'\subseteq\phi$, $|s'|>n^p$ and $\psi(\ell_{s'})= j$. Set $r'=\zeta(s')$, and observe that $r'\neq r$ by condition $(6)$. Finally, pick $q\in G$ such that $q\leq p$ and $n^q\geq |s'|$. It follows from condition $(12)$ that
$$
r'\re (F\times k_{n-1})=\zeta\big(t^{s'}_{\psi(\ell_{s'})}\big)\re (F\times k_{n-1}).
$$
Since $\zeta\big(t^{s'}_{\psi(\ell_{s'})}\big)=\zeta\big(t^{s'}_j\big)=\zeta(s)=r$, this shows that $r'$ is as desired. $\blacksquare$

\noindent\textbf{Claim 5.} $\pi_\alpha[Q']\subseteq z_\alpha\!\uparrow$ for every $\alpha\in\kappa$.

\noindent\textit{Proof.} Pick $\alpha\in\kappa$. We will use induction to show that $z_\alpha\subseteq\pi_\alpha\big(\zeta(\phi\re n)\big)$ for every $n\in\omega$ such that $\phi\re n\in\Ss$. The case $n=0$ is clear by our choice of $v_\alpha$. Now assume that $n\geq 1$ and $s=\phi\re n\in\Ss$. If $n\leq n^{p_\alpha}$ then $z_\alpha\subseteq\pi_\alpha\big(\zeta(s)\big)$ by our choice of $n_\alpha$, so assume that $n>n^{p_\alpha}$. Let $k\in z_\alpha$. We will show that $k\in\pi_\alpha\big(\zeta(s)\big)$ by considering the following cases:
\begin{itemize}
\item $k\in [0,k_{n-1})$,
\item $k\in [k_{n-1},k_n)$,
\item $k\in [k_n,\infty)$.
\end{itemize}
First assume that $k\in [0,k_{n-1})$. Pick $q\in G$ such that $q\leq p_\alpha$ and $n^q\geq n$. Since we have assumed that $n>n^{p_\alpha}$, condition $(12)$ guarantees that
$$
\zeta(s)\re (\{\alpha\}\times k_{n-1})=\zeta(s')\re (\{\alpha\}\times k_{n-1})
$$
for some $s'\in\Ss$ such that $s'\subsetneq s$. Therefore, the desired conclusion follows from the inductive hypothesis. The case $k\in [k_{n-1},k_n)$ is impossible, since $\phi(n-1)=s(n-1)=0$ implies that $[k_{n-1},k_n)\cap z_\alpha=\varnothing$ by our choice of $w_\alpha$. To conclude the proof, assume that $k\in [k_n,\infty)$. By picking $q\in G$ as above and using the assumption that $n>n^{p_\alpha}$, one sees that $v_\alpha\setminus k_n\subseteq\pi_\alpha\big(\zeta(s)\big)$ by condition $(13)$. Since $z_\alpha\subseteq v_\alpha$, it follows that $k\in\pi_\alpha\big(\zeta(s)\big)$. $\blacksquare$
\end{proof}

It seems worthwhile to state the following corollary\footnote{\,Who says that only \emph{theorems} should be allowed to have corollaries?} explicitly; however, in the proof of Theorem \ref{theorem_partial_answer}, we will need the stronger result given by Lemma \ref{lemma_generalized_miller}.

\begin{corollary}\label{corollary_generalized_miller}
Let $\kappa<\pppp$ be a cardinal, and let $\FF_\alpha$ be non-meager $\mathsf{P}$-filters for $\alpha\in\kappa$. Then $\prod_{\alpha\in\kappa}\FF_\alpha$ has the $\MP$.
\end{corollary}

The above corollary suggests the following question (see also Theorem \ref{theorem_product_miller}).

\begin{question}\label{question_product_miller}
Let $\kappa<\pppp$ be a cardinal, and let $X_\alpha$ be separable metrizable spaces with the $\MP$ for $\alpha\in\kappa$. Does it follow that $\prod_{\alpha\in\kappa}X_\alpha$ has the $\MP$?
\end{question}

We remark that, using Corollary \ref{corollary_generalized_miller}, it is possible to show that the intersection of fewer that $\pppp$ non-meager $\mathsf{P}$-filters is a non-meager $\mathsf{P}$-filter. We will omit the details however, since Corollary \ref{corollary_intersection_filters} gives a better result (better because $\pppp\leq\add(m^0)$ by \cite[Theorem 1.1]{goldstern_repicky_shelah_spinas}).

\section{A partial answer to Hru\v{s}\'{a}k's question}\label{section_partial_answer}

In this section, we will combine the results that we have discussed so far to finally prove one half of Conjecture \ref{conjecture_complete_answer}, as promised.

\begin{theorem}\label{theorem_partial_answer}
Let $\kappa$ be a cardinal and let $\FF_\alpha$ be filters for $\alpha\in\kappa$. If $\prod_{\alpha\in\kappa}\FF_\alpha$ is $\CDH$ then $\kappa<\pppp$ and each $\FF_\alpha$ is a non-meager $\mathsf{P}$-filter.
\end{theorem}
\begin{proof}
Set $X=\prod_{\alpha\in\kappa}\FF_\alpha$, and denote the projections by $\pi_\alpha:X\longrightarrow\FF_\alpha$ for $\alpha\in\kappa$. Assume that $X$ is $\CDH$. By \cite[Exercise 2.3.F.(c)]{engelking}, we must have $\kappa\leq\cccc$, otherwise $X$ would not be separable. It now follows from \cite[Theorem 8.1]{medini_steprans} that $\kappa<\pppp$.

It remains to show that each $\FF_\alpha$ is a non-meager $\mathsf{P}$-filter. By Theorem \ref{theorem_kunen_medini_zdomskyy}, it will suffice to show that each $\FF_\alpha$ has the $\MP$. In fact, since $X$ contains closed copies of each $\FF_\alpha$, it will be enough to show that $X$ has the $\MP$. First assume that $\kappa$ is finite. In this case, since $X$ is homeomorphic to a filter by Proposition \ref{proposition_generalized_shelah}, the desired conclusion follows from Theorem \ref{theorem_kunen_medini_zdomskyy}.

Now assume that $\kappa$ is infinite. In this case, it follows from Theorem \ref{theorem_dobrowolski_krupski_marciszewski} that $X$ is a Baire space, hence non-meager. Since each $\pi_\alpha$ is open, it follows that each $\FF_\alpha$ is non-meager. To see that $X$ has the $\MP$, pick a countable crowded $Q\subseteq X$. Fix countable dense subsets $D$ and $E$ of $X$ such that $D\supseteq Q$ and $E\subseteq\Cof^\kappa$. Since $X$ is $\CDH$, there exists a homeomorphism $h:X\longrightarrow X$ such that $h[D]=E$. Set $R=h[Q]\subseteq\Cof^\kappa$. Since every subset of $\Cof$ has $\omega$ as a pseudointersection, we can apply Lemma \ref{lemma_generalized_miller}, which yields a crowded $R'\subseteq R$ whose closure in $X$ is compact. To conclude the proof, set $Q'=h^{-1}[R']$, and observe that $Q'\subseteq Q$. Since $h$ is a homeomorphism, it is clear that $Q'$ is crowded and its closure in $X$ is compact.
\end{proof}

\section{Miller-measurability versus the Miller property}\label{section_measurability}

Given a separable metrizable space $Z$, we will say that $X\subseteq Z$ is \emph{Miller-null} if for every closed copy $N$ of $\omega^\omega$ in $Z$ there exists a closed copy $N'\subseteq N$ of $\omega^\omega$ in $Z$ such that $N'\cap X=\varnothing$; we will say that $X$ is \emph{Miller-full} if $Z\setminus X$ is Miller-null.\footnote{\,Of course, these notions are only of interest when $Z$ contains closed copies of $\omega^\omega$.} Using Theorem \ref{theorem_characterization_baire_space} plus standard arguments involving trees, one sees that the notions defined above are equivalent to the classical ones (see for example \cite{goldstern_repicky_shelah_spinas}) when $Z=\omega^\omega$. Recall that
$$
m^0=\{X\subseteq	\omega^\omega:X\text{ is Miller-null}\}
$$
is the \emph{Miller ideal} and that
$$
\add(m^0)=\mathsf{min}\left\{|\XX|:\XX\subseteq m^0\text{ and }\bigcup\XX\notin m^0\right\}
$$
denotes the \emph{additivity} of $m^0$. It is well-known that $\omega_1\leq\add(m^0)\leq\cccc$, where the first inequality follows from a standard fusion argument, and the second inequality is trivial. In fact, as we already mentioned at the end of Section \ref{section_old_lemma}, the inequality $\pppp\leq\add(m^0)$ holds.

Since the definition of the $\MP$ was inspired by an argument from \cite{miller}, it is not surprising that it is strongly connected to the above notions of Miller-measurability. This connection was already exploited in \cite{medini_zdomskyy}, but the following result describes it explicitly. Proposition \ref{proposition_characterization_miller} will be applied several times in the next two sections.

\begin{proposition}\label{proposition_characterization_miller}
Let $X$ be a separable metrizable space. Then the following conditions are equivalent:
\begin{itemize}
\item[(1)] $X$ has the $\MP$,
\item[(2)] For every metrizable compactification $\gamma X$ of $X$ and for every countable dense subset $D$ of $X$, the set $X\setminus D$ is Miller-full in $\gamma X\setminus D$,
\item[(3)] There exists a metrizable compactification $\gamma X$ of $X$ such that for every countable dense subset $D$ of $X$, the set $X\setminus D$ is Miller-full in $\gamma X\setminus D$.
\end{itemize}
\end{proposition}
\begin{proof}
In order to prove the implication $(1)\rightarrow (2)$, assume that condition $(1)$ holds. Fix a metrizable compactification $\gamma X$ of $X$; we will use $\cl$ to denote closure in $\gamma X$. Pick a countable dense $D\subseteq X$, and a closed copy $N$ of $\omega^\omega$ in $\gamma X\setminus D$. Set $Q=\cl(N)\cap D$, and observe that $Q$ is crowded by Lemma \ref{lemma_crowded}. Therefore, by the $\MP$, there exists a crowded $Q'\subseteq Q$ such that $\cl(Q')\subseteq X$. Set $N'=\cl(Q')\setminus D\subseteq X\setminus D$. Using Theorem \ref{theorem_characterization_baire_space}, one sees that $N'\subseteq N$ is a closed copy of $\omega^\omega$ in $\gamma X\setminus D$, as desired.

The implication $(2)\rightarrow (3)$ is trivial. In order to prove the implication $(3)\rightarrow (1)$, assume that condition $(3)$ holds. Pick a countable crowded $Q\subseteq X$. Let $D\supseteq Q$ be a countable dense subset of $X$. Using \cite[Exercise 7.4.17]{engelking}, it is possible to find a zero-dimensional Polish subspace $Z$ of $\gamma X$ such that $Z\cap D=Q$. Since Polish spaces have the $\MP$ by \cite[Proposition 2.3]{medini_zdomskyy}, there exists a compact subspace $K$ of $Z$ such that $K\cap Q$ is crowded and dense in $K$; we will use $\cl$ to denotes closure in $K$. Set $N=K\setminus Q$, and observe that $N$ is homeomorphic to $\omega^\omega$ by Theorem \ref{theorem_characterization_baire_space}. Furthermore, the fact that $N=(\gamma X\setminus D)\cap K$ shows that $N$ is closed in $\gamma X\setminus D$. It follows from condition $(3)$ that there exists a closed copy $N'\subseteq N$ of $\omega^\omega$ in $\gamma X\setminus D$ such that $N'\subseteq X\setminus D$. Set $Q'=\cl(N')\cap Q$, and observe that $Q'$ is crowded by Lemma \ref{lemma_crowded}. On the other hand $\cl(Q')\subseteq\cl(N')\subseteq N'\cup D\subseteq X$, which shows that $Q'$ has compact closure in $X$, as desired.
\end{proof}

We conclude this section by pointing out that the above characterization has some interesting consequences concerning non-meager $\mathsf{P}$-filters.

\begin{corollary}\label{corollary_miller_filters}
Let $\FF$ be a filter. Then the following conditions are equivalent:
\begin{itemize}
\item $\FF$ is a non-meager $\mathsf{P}$-filter,
\item $\FF\setminus D$ is Miller-full in $2^\omega\setminus D$ for every countable dense subset $D$ of $\FF$.
\end{itemize}
\end{corollary}
\begin{proof}
Simply combine Theorem \ref{theorem_kunen_medini_zdomskyy} and Proposition \ref{proposition_characterization_miller}.
\end{proof}

\begin{corollary}\label{corollary_intersection_filters}
Let $\kappa<\add(m^0)$ be a non-zero cardinal, and let $\FF_\alpha$ be non-meager $\mathsf{P}$-filters for $\alpha\in\kappa$. Then $\bigcap_{\alpha\in\kappa}\FF_\alpha$ is a non-meager $\mathsf{P}$-filter.
\end{corollary}
\begin{proof}
Using Theorem \ref{theorem_characterization_baire_space}, one sees that $2^\omega\setminus D$ is homeomorphic to $\omega^\omega$ for every countable dense subset $D$ of $2^\omega$. The rest of the proof is clear by Corollary \ref{corollary_miller_filters}.
\end{proof}

\section{Preservation of the Miller property under intersections}\label{section_preservation_intersections}

In this section, we will see that the $\MP$ is preserved by sufficiently small intersections. This shows that Corollary \ref{corollary_intersection_filters} is just a particular case of a more general phenomenon.

\begin{theorem}\label{theorem_intersection_miller}
Let $\kappa<\add(m^0)$ be a non-zero cardinal, let $Z$ be a separable metrizable space, and let $X_\alpha$ be subspaces of $Z$ with the $\MP$ for $\alpha\in\kappa$. Then $\bigcap_{\alpha\in\kappa}X_\alpha$ has the $\MP$.
\end{theorem}
\begin{proof}
Set $X=\bigcap_{\alpha\in\kappa}X_\alpha$, and pick a countable crowded $Q\subseteq X$. Fix a metrizable compactification $\gamma X$ of $X$. First, using \cite[Exercise 7.4.17]{engelking}, find a zero-dimensional Polish subspace $W$ of $\gamma X$ such that $Q\subseteq W$. Since Polish spaces have the $\MP$ by \cite[Proposition 2.3]{medini_zdomskyy}, there exists a compact subspace $K$ of $W$ such that $K\cap Q$ is crowded and dense in $K$; we will use $\cl$ to denote closure in $K$.

Set $X'=X\cap K$ and $X'_\alpha=X_\alpha\cap K$ for $\alpha\in\kappa$. Observe that each $X'_\alpha$ has the $\MP$ because it is a closed subspace of $X_\alpha$. Furthermore, it is clear that $K$ is a metrizable compactification of each $X'_\alpha$ because $K\cap Q\subseteq X'_\alpha$. It follows from Proposition \ref{proposition_characterization_miller} that each $X'_\alpha\setminus Q$ is Miller-full in $K\setminus Q$. Since $K\setminus Q$ is homeomorphic to $\omega^\omega$ by Theorem \ref{theorem_characterization_baire_space}, the assumption $\kappa<\add(m^0)$ ensures that $X'\setminus Q$ is Miller-full in $K\setminus Q$. In particular, there exists a closed copy $N$ of $\omega^\omega$ in $K\setminus Q$ such that $N\subseteq X'\setminus Q$. Set $Q'=\cl(N)\cap Q$, and observe that $Q'$ is crowded by Lemma \ref{lemma_crowded}. On the other hand $\cl(Q')\subseteq\cl(N)\subseteq N\cup Q\subseteq X$, which shows that $Q'$ has compact closure in $X$, as desired.
\end{proof}

Naturally, it would be nice to confirm that $\add(m^0)$ is optimal.

\begin{question}
Set $\kappa=\add(m^0)$. Are there a separable metrizable space $Z$ and subspaces $X_\alpha$ of $Z$ with the $\MP$ for $\alpha\in\kappa$ such that $\bigcap_{\alpha\in\kappa}X_\alpha$ does not have the $\MP$?
\end{question}

\section{Preservation of the Miller property under products}\label{section_preservation_products}

We conclude by showing that the $\MP$ is preserved by countable products (see also Question \ref{question_product_miller}). First we will focus on the following very special case, then we will apply Theorem \ref{theorem_intersection_miller} to obtain the full result.

\begin{lemma}\label{lemma_product_miller}
Let $X$ be a separable metrizable space with the $\MP$, and let $K$ be a compact metrizable space. Then $X\times K$ has the $\MP$.
\end{lemma}
\begin{proof}
Fix a metrizable compactification $\gamma X$ of $X$. We will show that condition $(3)$ of Proposition \ref{proposition_characterization_miller} holds for the compactification $\gamma X\times K$ of $X\times K$; we will use $\cl$ to denote closure in $\gamma X\times K$. So pick a countable dense subset $D$ of $X\times K$, and a closed copy $N$ of $\omega^\omega$ in $(\gamma X\times K)\setminus D$. Denote by $\pi:\gamma X\times K\longrightarrow\gamma X$ the projection, and set $E=\pi[D]$. Given $x\in E$, set $D_x=(\cl(N)\setminus N)\cap (\{x\}\times K)$, and observe that each $D_x\subseteq D$.

\noindent\textbf{Case 1.} There exists $x\in E$ such that $D_x$ is non-scattered.

\noindent Let $Q\subseteq D_x$ be crowded. Set $N'=\cl(Q)\setminus D$, and observe that $N'\subseteq N$ because $\cl(Q)\subseteq\cl(N)\subseteq N\cup D$. Furthermore, it is clear that $N'$ is closed in $(\gamma X\times K)\setminus D$. At this point, one sees that $N'$ is homeomorphic to $\omega^\omega$ by Theorem \ref{theorem_characterization_baire_space}. Finally, the fact that $\cl(Q)\subseteq\{x\}\times K\subseteq X\times K$ shows that $N'\subseteq (X\times K)\setminus D$.

\noindent\textbf{Case 2.} $D_x$ is scattered for every $x\in E$.

\noindent Since scattered separable metrizable spaces are Polish (see \cite[Corollary 21.21]{kechris}), we see that $(\{x\}\times K)\setminus D_x$ must be $\sigma$-compact for every $x\in E$. Hence
$$
N\cap (\{x\}\times K)=\big((\{x\}\times K)\setminus D_x\big)\cap\big(\cl(N)\cap(\{x\}\times K)\big)
$$
is $\sigma$-compact for every $x\in E$.

Set $Z=(\gamma X\setminus E)\times K$, and observe that $N\setminus (E\times K)$ is an analytic subset of the Polish space $Z$. Assume, in order to get a contradiction, that $N\setminus (E\times K)$ is contained in a $\sigma$-compact subspace of $Z$. Since $E$ is countable and $N\cap (\{x\}\times K)$ is $\sigma$-compact for every $x\in E$, it follows that $N$ itself is $\sigma$-compact, a contradiction. Therefore, by Theorem \ref{theorem_kechris_saint_raymond}, we can assume without loss of generality that $N$ is a (closed) subset of $Z$.

Now set $W=\gamma X\setminus E$, and observe that $\pi[N]$ is an analytic subset of the Polish space $W$. If $\pi[N]$ were contained in a $\sigma$-compact subspace of $W$, then it would easily follow that $N$ is $\sigma$-compact, a contradiction. Therefore, by Theorem \ref{theorem_kechris_saint_raymond}, there exists a closed copy $M$ of $\omega^\omega$ in $W$ such that $M\subseteq\pi[N]$. But $X$ has the $\MP$, so by Proposition \ref{proposition_characterization_miller} there exists a closed copy $M'\subseteq M$ of $\omega^\omega$ in $W$ such that $M'\subseteq X\setminus E$.

Finally, set $N'=N\cap\pi^{-1}[M']$, and observe that $N'$ is a closed subset of $Z$. Therefore, if $N'$ were contained in a $\sigma$-compact subspace of $Z$, then $N'$ itself would be $\sigma$-compact, contradicting the fact that $\pi[N']=M'$. It follows from Theorem \ref{theorem_kechris_saint_raymond} that there exists a closed copy $N''$ of $\omega^\omega$ in $Z$ such that $N''\subseteq N'$. It is easy to realize that $N''$ is a closed copy of $\omega^\omega$ in $(\gamma X\times K)\setminus D$ such that $N''\subseteq (X\times K)\setminus D$.
\end{proof}

\begin{theorem}\label{theorem_product_miller}
Let $\kappa\leq\omega$ be a cardinal, and let $X_\alpha$ be separable metrizable spaces with the $\MP$ for $\alpha\in\kappa$. Then $\prod_{\alpha\in\kappa}X_\alpha$ has the $\MP$.
\end{theorem}
\begin{proof}
Fix metrizable compactifications $\gamma X_\alpha$ of $X_\alpha$ for $\alpha\in\kappa$. Set $Z=\prod_{\alpha\in\kappa}\gamma X_\alpha$, and observe that $Z$ is separable metrizable because $\kappa\leq\omega$. Define
$$
X'_\alpha=\{x\in Z:x(\alpha)\in X_\alpha\}
$$
for $\alpha\in\kappa$, and observe that each $X'_\alpha$ has the $\MP$ by Lemma \ref{lemma_product_miller}. Since clearly $\prod_{\alpha\in\kappa}X_\alpha=\bigcap_{\alpha\in\kappa}X'_\alpha$, the desired conclusion follows from Theorem \ref{theorem_intersection_miller}.
\end{proof}

As in Section \ref{section_measurability}, we could formulate explicitly the particular case of Theorem \ref{theorem_product_miller} in which every $X_\alpha$ is a filter. We will refrain from doing so however, because this would amount to repeating Proposition \ref{proposition_generalized_shelah}, whose proof is much more direct.

We remark that the above result is somewhat relevant to a question of Fitzpatrick and Zhou \cite[Problem 4]{fitzpatrick_zhou}, which can be interpreted as asking for a ``natural'' topological property $\PP$ such that a zero-dimensional separable metrizable space $X$ has property $\PP$ iff $X^\omega$ is $\CDH$. Notice that any property which satisfies this requirement would have to be preserved by countable powers. In particular, Theorem \ref{theorem_product_miller} could be viewed as encouraging by anyone who thinks that the $\MP$ might be a good candidate for such a property.

\end{document}